\documentclass[10pt,reqno,oneside]{article}
\usepackage{amsfonts,amssymb,latexsym,xspace,epsfig,graphics,color}
\usepackage{amsmath,enumerate,stmaryrd,xy,amsthm}
\usepackage{bbm}
\usepackage{subfigure}
\usepackage{harvard}
\citationstyle{dcu}

\usepackage{tikz}

\newcommand{\E}{\mathbb{E}}

\newcommand{\loca}{S}

\newcommand{\N}{\ensuremath{\mathbb{N}}}

 \renewcommand{\P}{\ensuremath{\mathbb{P}}} 

\usepackage{algorithmic}
\usepackage{algorithm}

\newtheorem{lemma}{Lemma}

\newtheorem{theo}{Theorem}

\newtheorem{prop}{Proposition}
\newtheorem{coro}{Corollary}

\newtheorem{notalert}{Notation alert}[section]

\begin{document}



\title{Discrete one-dimensional coverage process on a renewal process}

\author{Sandro Gallo \\
{\small Departamento de Estat\'isitca}\\ {\small Universidade Federal de S\~ao Carlos}\\
{\small and}\\
Nancy L. Garcia\\{\small IMECC}\\{\small  Universidade Estadual de Campinas}}
\date{}

\maketitle

\begin{abstract}
We consider  the {following} coverage model on $\N$. For each site $i
\in \N$ we associate a pair $(\xi_i, R_i)$ where 
$\{\xi_0, \xi_1, \ldots \}$ is a 1-dimensional {undelayed} discrete renewal point process
and $\{R_0,R_1,\ldots\}$ is an i.i.d. sequence of
  $\N$-valued random variables. At each site where $\xi_i=1$ we start
  an interval of length $R_i$.  Coverage occurs if every site of
  $\N$ is covered by some interval. 
  We obtain sharp conditions for
both, positive and null probability of coverage.
  As corollaries, we extend
  results of the literature of rumor processes and discrete one-dimensional Boolean percolation.\end{abstract}

 \section{Introduction}

Consider two independent sequences of random variables, ${\boldsymbol \xi}=(\xi_i)_{i
  \ge 0}$ a sequence of Bernoulli random variables and $(R_i)_{i \ge
  0}$ a sequence of i.i.d. copies of some  $\N$-valued random variable
$R$. For each site $i \in \N$ associate the pair $(\xi_i,R_i)$, and
whenever $\xi_i=1$, start to the right an interval of length
$R_i$. Coverage occurs if every site of $\N$ is covered by some
interval. Naturally, the probability of coverage  depends on both,
the marginal distribution of $R$, and the joint distribution of the
$\xi_i$'s.  If the $\xi_i$'s are independent, this is a discrete
one-dimensional  Boolean percolation model, and the \emph{if and only
  if} condition  for positive probability of percolation is known, see
for example, \citeasnoun{bertacchi/zucca/2013} and
\citeasnoun{gallo/garcia/junior/rodriguez/2014}. For the Markov case,
\citeasnoun{athreya/roy/sarkar/2004} obtained sharp conditions for
both, positive and null probability of coverage.  {In the present
paper we extend their results for  undelayed renewal sequences.} This in particular answers an open question left in the recent review on the subject \cite[see problem (ii) in Section 7 therein]{junior/machado/ravi/2016}.

\vspace{0.2cm}

Our strategy of proof is to study the dual process of the coverage model. We prove that the dual process is  in fact a renewal process, with an inter-arrival distribution that depends on the distributions of $R$ and of the renewal sequence ${\boldsymbol \xi}$. This  result allows to obtain a closed formula giving the exact probability of coverage. Informally, we could say that  this formula is expressed in terms of the generating function of arbitrarily large vectors of the renewal process ${\boldsymbol \xi}$. With this in hands, explicit conditions  will be obtained. The condition for null probability of coverage is easily obtained while the condition for positive probability of coverage is more tricky, involving exponential inequalities from martingales differences.

\vspace{0.2cm}

It is worth mentioning that in \citeasnoun{junior/machado/zuluaga/2011},
\citeasnoun{bertacchi/zucca/2013}, and \citeasnoun{gallo/garcia/junior/rodriguez/2014}, the one-dimensional discrete Boolean percolation model is referred as \emph{firework process}, and it is given an interpretation in terms of information propagation. Initially, there is one spreader at site 0 and ignorants at all (or some random subset of) the other sites of $\N$. { Each time an ignorant \emph{hears} the rumor, it becomes a spreader.} Each spreader tries to spread the information to the individuals on its right  within a random distance.  These works also analyse a closely related model, called \emph{reverse firework process}, which also starts with a spreader at $0$, but then, each ignorant of $\N$ tries to take the information, independently from the others, from  spreaders within a random distance on its left.  In both models, the information survives if the propagation lasts forever, reaching infinity. In all these works, the location of the ignorants on $\N$, if random, was described by an i.i.d. Bernoulli sequence indicating whether or not there is an individual at each site.

 In these terms, our model is an extension of the firework process in which the individuals are localized according to a binary undelayed renewal sequence ${\boldsymbol \xi}$.  In our terminology, {propagation} means that there is coverage. 
We will also prove  that the reverse firework process is nothing more than the {dual} of the firework process, showing that both processes are indeed, two sides of the same coin.

\vspace{0.2cm}

The rest of the paper is organized as follows. In the next section, we formally introduce  the coverage model and state the results. The proofs are given in Section \ref{sec:proofs}. Section \ref{sec:discussion} contains a discussion of the results, specifically about a technical condition and the obtention of explicit bounds for the probability of coverage and the relation with the rumor processes described above. We conclude the paper with an appendix containing some results concerning renewal theory used in our proofs and a technical lemma based on techniques involving martingale difference.

\section{Model and results} \label{sec:model}

\subsection{The model}

Let $(\Omega,\mathcal{F}, \P)$ be the single probability space in which all the forthcoming  random variables are defined.
Let $(T_i)_{i\ge1}$ be a sequence of independent copies of some  $\N^\star$-valued r.v. $T$ ({in this paper we use $\N=\N^\star\cup\{0\}$).} Let us also introduce the binary process ${\boldsymbol
  \xi}=(\xi_i)_{i \ge0}$ as $\xi_0=1$ and for $i\ge1$, let  $\xi_i=1$ if, and only if, there exists $n\ge1$ such that $\sum_{k=1}^nT_k=i$.
${\boldsymbol \xi}$ is a binary undelayed renewal sequence with inter-arrival $T$ \cite{feller/1968,lindvall/1992}.  
Finally, let $(R_i)_{i\in\N}$ be a sequence of independent copies of some $\N$-valued random variable $R$ satisfying $\P(R=0)>0$, and independent of the sequence ${\boldsymbol
  \xi}$.
 
The coverage process is now defined simply as follows.  Consider the sequence ${\cal C}_i,\,i\ge0$ of random intervals defined as $
{\cal C}_i:=[i+1,i+R_i]$ if $\xi_i=1$ and ${\cal C}_i=\emptyset$ if $\xi_i=0$ or $R_i=0$. 
We say that there is \emph{coverage} if the event
\[
\mathcal{A}:=\{\cup_{i\ge0}{\cal C}_i=\N\}
\]
occurs. The main objective of the present paper is to study $\P(\mathcal{A})$. We begin with the  following simple result.

\begin{prop}\label{prop:obvio}
If $\E T=\infty$, then $\P(\mathcal{A})=0$.
\end{prop}

Therefore, the case where $\E T=\infty$ is not interesting for this model.  Unless explicitly mentioned, we will always assume  $\E T<\infty$ (i.e. the process $(\xi_i)_{i \ge 0}$ is positive recurrent) in the sequel. We further assume, to simplify the presentation, that ${\boldsymbol\xi}$ is aperiodic, that is, the period $d:=\textrm{gcd}\{k\ge1:\P(T=k)>0\}$ equals $1$.

Our first main result is a closed formula for the probability of
coverage (use $1/\infty=0$). 
\begin{theo}\label{theo} $\P(\mathcal{A})=\left(1+\sum_{n\ge1}\E\prod_{i=0}^{n-1}\alpha_i^{\xi_{i+1}}\right)^{-1}$ where $\alpha_n:=\P(R\leq n),\, n\ge0$.
\end{theo}

Observe that the quantity 
\[
\E\prod_{i=0}^{n-1}\alpha_i^{\xi_{i+1}}=\sum_{(a_1,\ldots,a_{n})\in\{0,1\}^n}\alpha_0^{a_1}\ldots
\alpha_{n-1}^{a_{n}}\P(\xi_1=a_1,\ldots,\xi_{n}=a_{n})\] 
is  the probability generating function of the binary vector $(\xi_1,\ldots,\xi_{n})$ at $(\alpha_0,\ldots,\alpha_{n-1})$. This quantity is difficult to handle in general. Naturally, when ${\boldsymbol\xi}$ is an i.i.d. sequence, we   use independence to factorize $\E\prod_{i=0}^{n-1}\alpha_i^{\xi_{i+1}}$, and a direct consequence of Theorem \ref{theo} is that 
\begin{equation}\label{coro:i.i.d}
\P(\mathcal{A})=\left(1+\sum_{n\ge1}\prod_{i=0}^{n-1}[1-p(1-\alpha_i)]\right)^{-1}
\end{equation}
where $p=\P(\xi_i=1),i\ge1$.
%
%
This result was obtained by \citeasnoun{gallo/garcia/junior/rodriguez/2014}, and so far, such an exact expression is only available in the i.i.d. case.  We postpone to Section \ref{sec:discussion} the discussion concerning possible upper and lower bounds for $\P(\mathcal{A})$ in non-i.i.d. cases.

Our next results give explicit sufficient conditions for  null or positive probability of coverage.  For null probability, the following result holds without any further assumptions concerning $\boldsymbol\xi$. 
\begin{prop}\label{coro:1} If $\limsup_{n\rightarrow\infty} n(\E T)^{-1}(1-\alpha_n)<1$, then $ \P(\mathcal{A})=0$.
\end{prop}
To guarantee positive probability of coverage, we will need an
  extra assumption concerning ${\boldsymbol\xi}$. Let us introduce the notation
  \[q_i:=\P(T\ge i+2|T\ge i+1),\quad i\ge0\]
  and
  \begin{equation}\label{eq:q*}
  q_i^\star:=\max_{n\leq i}q_n,\,\,i\ge0.
  \end{equation}
\begin{prop}\label{coro:2} If  
\begin{equation}
\label{eq:ok}
\sum_{j=1}^k\prod_{i=k}^{k+j-2} q^\star_i =
 o(k).
\end{equation}
and  $\liminf_{n\rightarrow\infty} n(\E T)^{-1}(1-\alpha_n)>1$, then $\P(\mathcal{A})>0$.
\end{prop}

Propositions \ref{coro:1} and \ref{coro:2} together yield a tight phase transition between null and positive probability of coverage for processes satisfying \eqref{eq:ok}. Let us also observe that we know that condition \eqref{eq:ok} is not optimal, a more general assumption will be given later in the proof (condition \eqref{eq:general}), and we will discuss in Section \ref{sec:discussion} the difference between both conditions.

We conjecture that $\liminf_{n\rightarrow\infty} n(\E T)^{-1}(1-\alpha_n)>1$ implies $\P(\mathcal{A})>0$ without any further assumption (only positive recurrence) on the renewal process, but we have not been able to prove this.

  Some explicit example satisfying condition \eqref{eq:ok} are listed below. 
\begin{enumerate} \item {\it The Doeblin case:} $q_i\leq
1-\epsilon$ for $i\ge0$ clearly satisfies condition \eqref{eq:ok} since the lefthand side is bounded in $k$. 
\item {{\it The Markov case:} $q_i=q_1$ for any $i\ge1$, with $q_0$, $q_1$ $\in(0,1)$. In this case,  ${\boldsymbol\xi}$ is the positive recurrent Markov chain with transition matrix $P(1|1)=1-q_0$ and $P(1|0)=1-q_1$ and we recover the results obtained by \citeasnoun{athreya/roy/sarkar/2004}. Notice that this case is a particular case of the preceding (Doeblin) assumption.}
\item {\it The monotonicity case:}  $q_i=1-\frac{1}{i^\beta}$, $i\ge2$ with $0<\beta<1$,  satisfies condition (\ref{eq:ok}), as we will show in Section \ref{sec:discussion}.
\end{enumerate}

 Section \ref{sec:discussion}  contains other new results, comparisons with the literature, and explicit examples.

\section{Proofs of the results}\label{sec:proofs}

The present section is organised as follows.   In Section \ref{sec:3.1} we introduce and analyse  a \emph{dual process} that we will use in Section \ref{sec:3.3} for the proofs of Theorem
\ref{theo} and Propositions \ref{coro:1} and \ref{coro:2}. {Some simple facts about renewal sequences, that we will use along our proofs, are presented in Appendix \ref{sec:renewal}. {The proof of Lemma \ref{lem:concentration1}, which  is rather lengthy, is deferred to Appendix \ref{sec:concentration}}.}

\subsection{The dual process} \label{sec:3.1}

{Consider our coverage process, and for any $i\ge1$ and $j\ge0$, { say that $i$ is \emph{reachable from $j$}} if there exists $k\in\mathbb{N}^\star$ and a sequence of natural numbers $j=n_0<n_1<\ldots<n_k= i$ {such that  $n_{l}\in \mathcal{C}_{n_{l-1}}$ for $l=1,\ldots,k$.}


For each $n\ge1$ let ${\bf Y}^{(n)}$ be the binary process,  defined through $Y_0^{(n)}=\xi_n$ and for $i=1,\ldots,n$
\[
Y_i^{(n)}={\bf 1}\{\textrm{$n$ is reachable from $n-i$}\}.
\]
It can be seen as a dual to our coverage process. { Figure \ref{FIG:FP} gives a pictorial representation of the coverage process and its dual $Y^{(n)}_i,i=0,\ldots,n$ with $n=10$. For instance we have  $(Y^{(n)}_0,\ldots,Y_n^{(n)})=(0,0,1,0,0,1,0,0,0,0,1)$.
  The underlying renewal process of the upper part is ${\boldsymbol\xi}$, undelayed, and the one of the lower part is $\boldsymbol\zeta$, with delay of size $1$ in the present case since $\xi_n=0$ and $\xi_{n+1}=1$.}

  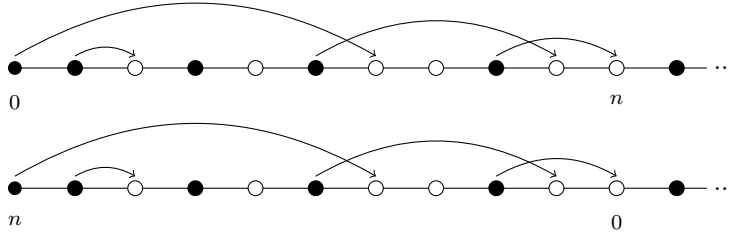
\begin{figure}[h]

\begin{center}
\begin{tikzpicture}[scale=0.8]

\draw (-3,-2) -- (8.5,-2);
\draw (-3,-2.3) node[below,font=\footnotesize] {$0$};
\draw (7,-2.3) node[below,font=\footnotesize] {$n$};
\draw (8.8,-2) node {...};


\filldraw [black] (-3,-2) circle (3pt);
\draw [very thick] (-2,-2) circle (3pt);
\filldraw [black] (-2,-2) circle (3pt);
\draw [very thick] (-1,-2) circle (3pt);
\filldraw [white] (-1,-2) circle (3pt);
\draw [very thick] (0,-2) circle (3pt);
\filldraw [black] (0,-2) circle (3pt);
\draw [very thick] (1,-2) circle (3pt);
\filldraw [white] (1,-2) circle (3pt);
\draw [very thick] (2,-2) circle (3pt);
\filldraw [black] (2,-2) circle (3pt);
\draw [very thick] (3,-2) circle (3pt);
\filldraw [white] (3,-2) circle (3pt);
\draw [very thick] (4,-2) circle (3pt);
\filldraw [white] (4,-2) circle (3pt);
\draw [very thick] (5,-2) circle (3pt);
\filldraw [black] (5,-2) circle (3pt);
\draw [very thick] (6,-2) circle (3pt);
\filldraw [white] (6,-2) circle (3pt);
\draw [very thick] (7,-2) circle (3pt);
\filldraw [white] (7,-2) circle (3pt);
\draw [very thick] (8,-2) circle (3pt);
\filldraw [black] (8,-2) circle (3pt);


\draw [->] (-3,-1.8) to [bend left=30] (3,-1.8);
\draw [->] (2,-1.8) to [bend left=30] (6,-1.8);
\draw [->] (-2,-1.8) to [bend left=30] (-1,-1.8);
\draw [->] (5,-1.8) to [bend left=30] (7,-1.8);


\draw (-3,-4) -- (8.5,-4);
\draw (-3,-4.3) node[below,font=\footnotesize] {$n$};
\draw (7,-4.3) node[below,font=\footnotesize] {$0$};
\draw (8.8,-4) node {...};


\filldraw [black] (-3,-4) circle (3pt);
\draw [very thick] (-2,-4) circle (3pt);
\filldraw [black] (-2,-4) circle (3pt);
\draw [very thick] (-1,-4) circle (3pt);
\filldraw [white] (-1,-4) circle (3pt);
\draw [very thick] (0,-4) circle (3pt);
\filldraw [black] (0,-4) circle (3pt);
\draw [very thick] (1,-4) circle (3pt);
\filldraw [white] (1,-4) circle (3pt);
\draw [very thick] (2,-4) circle (3pt);
\filldraw [black] (2,-4) circle (3pt);
\draw [very thick] (3,-4) circle (3pt);
\filldraw [white] (3,-4) circle (3pt);
\draw [very thick] (4,-4) circle (3pt);
\filldraw [white] (4,-4) circle (3pt);
\draw [very thick] (5,-4) circle (3pt);
\filldraw [black] (5,-4) circle (3pt);
\draw [very thick] (6,-4) circle (3pt);
\filldraw [white] (6,-4) circle (3pt);
\draw [very thick] (7,-4) circle (3pt);
\filldraw [white] (7,-4) circle (3pt);
\draw [very thick] (8,-4) circle (3pt);
\filldraw [black] (8,-4) circle (3pt);


\draw [->] (-3,-3.8) to [bend left=30] (3,-3.8);
\draw [->] (2,-3.8) to [bend left=30] (6,-3.8);
\draw [->] (-2,-3.8) to [bend left=30] (-1,-3.8);
\draw [->] (5,-3.8) to [bend left=30] (7,-3.8);

\end{tikzpicture}
\end{center}
\caption{The coverage process (upper part) and its dual $Y^{(n)}_i,i=0,\ldots,n$  (lower part). The black marks indicate $1$'s in the renewal sequences.}\label{FIG:FP}
\end{figure}

The following proposition will be an essential tool for our proofs.

\begin{prop}\label{theo:rfp}
For any $n\ge1$, $Y_i^{(n)},i=0,\ldots,n$ are the $n$ first random variables of a delayed renewal sequence with aperiodic inter-arrival  $T_{\bf Y}$ having distribution 
\[
\P(T_{\bf Y}\ge k)=\E\prod_{i=1}^{k-1}\alpha_{i-1}^{\xi_i}\,\,,\quad k\ge 2.
\]
 {It is recurrent if, and only if, $\prod_{i\ge0} \alpha_i=0$.} 
\end{prop}
\begin{proof}[Proof of Proposition \ref{theo:rfp}] 
We will omit the upper index ``$(n)$'' since  $n$ is fixed along this proof.
For $i=0,\ldots,n$, let $\bar R_i:=R_{n-i}$, which are also independent copies of $R$. The sequence $\zeta_i:=\xi_{n-i}$ is  a renewal sequence with the same inter-arrival distribution as ${\boldsymbol\xi}$, and undelayed  if $\xi_n=1$. From the definition of the dual,
\begin{align}\label{eq:conta}
\{Y_{i}=1\}&=\{\zeta_i=1\}\cap\{\bar R_{i}\ge\ell_1(Y_{0}^{i-1})+1\wedge i\},
\end{align}
where $\ell_1(Y_0^{i}):=\inf\{j\ge0: \,Y_{i-j}=1\}$ (we use the  condensed notation $a_m^{n}$, $m\le n$ for the string $(a_m,\ldots, a_{n})$).

Using independence between the $\bar R_j$'s and the 
$\zeta_j$'s, we have for any $i=1,\ldots,n$
\begin{align*}
\mathbb{E}({\bf1}\{Y_{i}=1\}|Y_{0}^{i-1})&=\P(\bar R_{i}\ge\ell_1(Y_{0}^{i-1})+1)\E\left({\bf1}\{\zeta_i=1\}|Y_{0}^{i-1}\right).
\end{align*}
On the other hand, $\E({\bf1}\{\zeta_n=1\}|\zeta_0^{n-1})=\E({\bf1}\{\zeta_n=1\}|\ell_1(\zeta_0^{n-1}))$, thus, using the tower property and the fact that $Y_i=1$ implies $\zeta_i=1$, we have
\[
\E\left({\bf1}\{\zeta_i=1\}|Y_{0}^{i-1}\right)=\E\left({\bf1}\{\zeta_i=1\}|\ell_1(Y_{0}^{i-1})\right),
\]
characterising $Y_i$,  $i=0,\ldots,n$, as a renewal sequence. This renewal sequence has a delay if $Y_0\ne1$ (equivalently, $\xi_n\ne1$). Thus the elapsed time until the first occurrence of a $1$ may differ from the elapsed time between two consecutive occurrences of $1$.

We now compute the distribution of the inter-arrival time $T_{\bf Y}$, 
\[
\P(T_{\bf Y} \ge k) =\P(Y_1^{k-1}=0_1^{k-1}|Y_0=1)=\P\left(\bigcup_{a_{1}^{k-1}\in\{0,1\}^{k-1}}     A(a_{1}^{k-1}) \Big| Y_0=1 \right),
\]
where
\[
A(a_{1}^{k-1}) = \{\zeta_{1}^{k-1}=a_{1}^{k-1}\}\cap\bigcap_{i=1}^{k-1}\left\{\bar R_i <i.{\bf 1}\{a_i=1\}+\infty.{\bf1}\{a_i=0\}\right\},
  \]
 is the event that represents that, at each site $i\in\{1,\ldots,k-1\}$ such that $\zeta_i=1$,
we have $\bar R_i <  i$, the symbol $\infty$  means that if
$\zeta_i=0$, the value of $\bar R_i$ does not matter. Thus, using independence
\begin{align*}
\P(T_{\bf Y} \ge k)&=\sum_{a_{1}^{k-1}\in\{0,1\}^{k-1}} \P(\zeta_{1}^{k-1}=a_{1}^{k-1}|\zeta_{0}=1)\alpha_0^{a_1}\ldots \alpha_{k-2}^{a_{k-1}}\\&=\E\left(\prod_{i=1}^{k-1}\alpha_{i-1}^{\zeta_i}\big|\zeta_0=1\right)\\&=\E\prod_{i=1}^{k-1}\alpha_{i-1}^{\xi_i}.
\end{align*}
In the last equality above, recall that the sequence $\xi_i,i\ge0$ starts with $\xi_0=1$, and this sequence has the same distribution as $\zeta_i,i\ge0$ started with $\zeta_0=1$.

To see that it is aperiodic,  observe that 
$\P(T_{\bf Y}=k)=\E[(1-\alpha_{k-1}^{\xi_{k}})\prod_{i=1}^{k-1}\alpha_{i-1}^{\xi_{i}}]\ge\prod_{i=0}^{k-2}\alpha_i\P(\xi_k=1)(1-\alpha_{k-1})$ which is positive if $\P(T=k)>0$.  This means that the period of ${\bf Y}$ is smaller or equal to that of ${\boldsymbol\xi}$, and the latter is aperiodic, thus this is also the case of the former.

{We now prove  recurrence. The renewal sequence is recurrent if, and only if, $
\lim_k\P(T_{\bf Y}\ge k)=0$, that is $\E\prod_{i\ge0}\alpha_i^{\xi_{i+1}}=0.$
Recall the sequence $(T_i)_{i\ge1}$ defined at the beginning of Section \ref{sec:model} and define $(S_n)_{n\ge0}$ through $S_0:=0$ and, for $n\ge1, S_n:=\sum_{k=1}^nT_k$.  Observe that $\xi_i={\bf1}\{\exists n\ge0:S_n=i\}$, $i\ge0$, thus
\[
\E\prod_{i\ge0}\alpha_i^{\xi_{i+1}}=\E\prod_{i\ge0}\alpha_{S_{i+1}}.
\]
The variables $S_1,S_2-S_1,\ldots,S_i-S_{i-1},\ldots$ are i.i.d., and the sequence $\alpha_k$ is a non-decreasing function of $k$, 
thus we can conclude the proof  using Theorem 3(iii) in \citeasnoun{puri/1978} which states that
\[
\E\prod_{i\ge0}\alpha_{S_{i+1}}=0\Leftrightarrow \sum_{k\ge0}(1-\alpha_k)=\infty.
\]
It is well-known that the latter is equivalent to $\prod_{i\ge0}\alpha_i=0$ (recall that $\alpha_0\in(0,1]$).}

\end{proof}

\subsection{Proofs of the results} \label{sec:3.3}

\begin{proof}[Proof of Proposition \ref{prop:obvio}]
{Recall from the  proof of Proposition \ref{theo:rfp} that $\{Y_n=1\}\subset \{\zeta_n=1\}$.  We have $\P(\mathcal{A})=\lim_n\P(Y_n=1)\leq \lim_n\P(\zeta_n=1)$ which vanishes as  $n$ diverges when $\E T=\infty$ by the  Renewal Theorem.}
\end{proof}

\begin{proof}[Proof of Theorem \ref{theo}]
We divide the proof into two cases. 
\begin{itemize}
\item  Case 1: $\prod_{i\ge0}\alpha_i>0$. In this case coverage occurs with probability 0 even when $\xi_i=1$ for any $i\ge1$ (see for instance \citeasnoun{junior/machado/zuluaga/2011}).  {Since $\sum_{n\ge1}\E\prod_{i=0}^{n-1}\alpha_i^{\xi_{i+1}}\ge \sum_{n\ge1}\prod_{i=0}^{n-1}\alpha_i=\infty$ and we have $\P(\mathcal{A})=0$. }
\item  Case 2: $\prod_{i\ge0}\alpha_i=0$.  {Notice that 
\[\P(\mathcal{A})\,=\,\P\left(\cap_n\{\textrm{$n$ is reachable from 0}\}\right) \,=\lim_n\P(\textrm{$n$ is reachable from 0})\]
since $\{\{\textrm{$n$ is reachable from 0}\}, n \ge 1\}$ is a decreasing sequence of sets. } Thus $\P(\mathcal{A})=\lim_n\P(Y^{(n)}_n=1)$. 
As we said, $Y^{(n)}_n$ is delayed, in fact, the random variable $r_n({\boldsymbol\xi}):=\inf\{i\ge0:\xi_{n+i}=1\}$ will specify the distribution of this delay. So we can write 
\begin{align*}
\lim_n\P(Y^{(n)}_n=1)&=\lim_n\sum_{i\ge0}\P(Y^{(n)}_n=1|r_{n}(\boldsymbol\xi)=i)\P(r_{n}(\boldsymbol\xi)=i)\\
&=\sum_{i\ge0}\lim_n\P(Y^{(n)}_n=1|r_{n}(\boldsymbol\xi)=i)\lim_n\P(r_{n}(\boldsymbol\xi)=i).
\end{align*}
Now, recall that $\boldsymbol\xi$ is positive recurrent, and thus $\P(r_{n}(\boldsymbol\xi)=i)$ converges to the stationary measure of the string $0\ldots01$ of size $i$ (which reduces to symbol $1$ when $i=0$). Denoting it by $\mu_i$, we have in particular $\sum_{i\ge0}\mu_i=1$.  On the other hand, for each $i\ge0$, the renewal sequence $Y^{(n)}_j,j=0,\ldots,n$ has a fixed delay distribution.  No matter what this distribution is, the Renewal Theorem guarantees that $\P(Y^{(n)}_n=1|r_{n}(\boldsymbol\xi)=i)$ converges to 
$(\E T_{\bf Y})^{-1}$. Thus we conclude that
$$\P(\mathcal{A})=\frac{1}{1+\sum_{k\ge1}\E\left[\prod_{i=0}^{k-1}\alpha_i^{\xi_{i+1}}\right]}.$$ 
\end{itemize}\end{proof}

\begin{proof}[Proof of Proposition \ref{coro:1}]
An upper bound  for $\P(\mathcal{A})$ follows easily from Theorem \ref{theo} and Jensen inequality
\begin{align*}
\E\prod_{i=0}^{k-1}\alpha_i^{\xi_{i+1}}&\ge\prod_{i=0}^{k-1} \alpha_i^{\P(\xi_{i+1}=1)}.
\end{align*}
Therefore, 
\begin{equation}
\P(\mathcal{A})\le \left(1+\sum_{k\ge0}\prod_{i=0}^{k-1} \alpha_i^{\P(\xi_{i+1}=1)}\right)^{-1}
\end{equation}
and a sufficient condition for a.s. extinction is that $\sum_{k\ge0}\prod_{i=0}^{k-1} \alpha_i^{\P(\xi_{i+1}=1)}=\infty$. The Raabe criterion states, for non-negative real sequences $u_k,\,k\ge1$, that if $\limsup k\left(u_k/u_{k+1}-1\right)<1$, then $\sum_ku_k=\infty$. In our case, it is therefore enough that
\[
\limsup\, k\left(\frac{1}{\alpha_{k}^{\P(\xi_{k+1}=1)}}-1\right)<1.
\]
By the Renewal Theorem, $\P(\xi_{k+1}=1)$ converges to $(\E T)^{-1}$. Moreover, $\alpha_k$ increases to $1$. Thus,  it is enough that
\[
\limsup\, k(1-\alpha_{k}^{(\E T)^{-1}})<1.
\]
Using Taylor's expansion, for large $k$, $\alpha_{k}^{(\E T)^{-1}}= 1-(\E T)^{-1}(1-\alpha_k)\left[1-O(1-\alpha_k)\right]$ and we conclude that a sufficient condition for a.s. extinction is
\[
\limsup\, k(\E T)^{-1}(1-\alpha_{k})<1.
\]
\end{proof}

Once we get a lower bound for $\P(\mathcal{A})$, the proof of  Proposition \ref{coro:2} follows the same lines as the proof of Proposition \ref{coro:1}. However, it is trickier to get a useful lower bound in general. We present such a lower bound in Lemma \ref{lem:concentration1}, whose proof, rather lengthy, follows essentially the same arguments of the
concentration inequalities obtained in
\citeasnoun{chazottes/collet/kulske/regig/2007}. However, there are
some slight differences which are of substantial importance and make
the relation less explicit, thus, we will give the proof in Appendix
\ref{sec:concentration}. Before stating the lemma, we need to introduce a coupling of renewal sequences starting from different delays. For any pair of integers $i,j\ge 0$,  we will denote by $(\tilde{\boldsymbol\xi}^{(i)},\tilde{\boldsymbol\xi}^{(j)})$  a coupling between two renewal processes with the same inter-arrival distribution, one with a delay of size $i$ and the other with a delay of size $j$. It is convenient to define this coupling as a coordinate-wise function of a coupling $(\tilde{\boldsymbol\zeta}^{(i)},\tilde{\boldsymbol\zeta}^{(j)})$ of their backward recurrence time chains on $\N$ \cite[3.3.1]{meyn/tweedie/2009} with common transition matrix $Q$
\[
Q(n,n+1)=q_n=1-Q(n,0), \,\,\,n\ge0,
\]
and starting from states $i$ and $j$ respectively.  The coupling uses  a single sequence of i.i.d. random variables $(U_i)_{i\ge1}$ uniformly distributed in $[0,1[$ ({let $\P_{\bf U}$ denote the law of this sequence}) and is defined iteratively as follows:
 $\left(\tilde\zeta^{(i)}_{0},\tilde\zeta^{(j)}_{0}\right)=(i,j)$ and 
\[
\left(\tilde\zeta^{(i)}_{k},\tilde\zeta^{(j)}_{k}\right)=\left((\tilde\zeta^{(i)}_{k-1}+1){\bf1}\left\{ U_k{\le} q_{\tilde\zeta^{(i)}_{k-1}}\right\},(\tilde\zeta^{(j)}_{k-1}+1){\bf1}\left\{ U_k{\le} q_{\tilde\zeta^{(j)}_{k-1}}\right\}\right), 
\]
for any $k\ge1$.  Then, we naturally obtain $(\tilde{\boldsymbol\xi}^{(i)},\tilde{\boldsymbol\xi}^{(j)})$ by defining for any $k\ge0$
\[
\left(\tilde\xi^{(i)}_{k},\tilde\xi^{(j)}_{k}\right)=\left({\bf1}\{\tilde\zeta^{(i)}_{k}=0\},{\bf1}\{\tilde\zeta^{(j)}_{k}=0\}\right).
\]
Based on this coupling, we now define the random variable 
\[
\tau_{i,j}=\inf\{l\ge1:\tilde\xi^{(i)}_l=\tilde\xi^{(j)}_l=1\} = \inf\{l\ge1:\tilde\zeta^{(i)}_l=\tilde\zeta^{(j)}_l=0\}.
\]

\begin{lemma}\label{lem:concentration1}
We have
\begin{align*}
\E\prod_{i=0}^{k-1}\alpha_i^{\xi_i}\leq e^{C_k}\prod_{i=0}^{k-1}\alpha_i^{\P(\xi_i=1)}
\end{align*}
where 
\[
C_k:=\frac{1}{8}\sum_{i=1}^k\left(\sum_{j=i}^k\sup_{l=1,\ldots,i}\P_{\bf U}(\tau_{0,l}\ge j)|\log\alpha_{j-1}|\right)^2.
\]
\end{lemma}

We are now ready to prove Proposition \ref{coro:2}.

\begin{proof}[Proof of Proposition \ref{coro:2}]
We have to prove that the sum in $k$ in the rhs of Lemma \ref{lem:concentration1} is
summable. The Raabe criterion states, for non-negative real sequences $u_k,\,k\ge1$, that if $\liminf k\left(u_k/u_{k+1}-1\right)>1$, then $\sum_ku_k<\infty$. We already know that
$\liminf_{n\rightarrow\infty} n(\E T)^{-1}(1-\alpha_n)>1$ (which is
the assumption of the proposition) makes
$\sum_{k\ge1}\prod_{i=0}^{k-1}\alpha_i^{\P(\xi_{i+1}=1)}$
summable. Thus, to conclude the statement of the lemma, it only remains to prove that, under this assumption, 
\begin{equation}\label{eq:vanishes}C_k-C_{k+1}\rightarrow0.\end{equation}

Consider  any sequence $\alpha_n,\,n\ge0$, satisfying 
$\liminf_{n\rightarrow\infty} n(\E T)^{-1}(1-\alpha_n)>1$ 
as specified by the corollary. {Then, there exists some $K$ such that  $\alpha_{k} < 1 - \E T /k$ for $k\ge K$. Therefore, it is enough to prove that \eqref{eq:vanishes} vanishes with the sequence $\bar\alpha_k$ defined through
  $\bar\alpha_{k}=\alpha_k$ for $k<K$ and $\bar\alpha_{k}= 1- \E T/k$ for $k\ge K$. This is because if coverage occurs with
  positive probability with the sequence $P(\bar R\le k)=\bar\alpha_k$, $k\ge0$, by an
  obvious coupling argument (using $R_i\ge \bar R_i$, $i\ge0$) it
  survives also with the sequence $P(R\le k)=\alpha_k$, $k\ge0$. 
  
  So we now prove that, under condition \eqref{eq:ok}, \eqref{eq:vanishes} holds with $\bar\alpha_k,k\ge0$ instead of $\alpha_k,k\ge0$. Let
  \[
A_{i,k}:=\sum_{j=i}^k\sup_{l=1,\ldots,i}\P_{\bf U}(\tau_{0,l}\ge j)|\log\bar\alpha_{j-1}|.
\]
Simple calculations show that
\begin{align*}
C_k-C_{k+1}&=\sum_{i=1}^k(A_{i,k}-A_{i,k+1})(A_{i,k}+A_{i,k+1})-\left(\sup_{l=1,\ldots,k+1}\P_{\bf U}(\tau_{0,l}\ge k+1)|\log\bar\alpha_{k+1}|\right)^2.
\end{align*}
By the choice of $\bar\alpha_k$, it is enough to prove that $\sum_{i=1}^k(A_{i,k}-A_{i,k+1})(A_{i,k}+A_{i,k+1})$ vanishes. Observe that
\[
A_{i,k}-A_{i,k+1}=-\sup_{l=1,\ldots,i}\P_{\bf U}(\tau_{0,l}\ge k+1)|\log\bar\alpha_{k}|\]
and that
\[
A_{i,k}+A_{i,k+1}=2\sum_{j=i}^k\sup_{l=1,\ldots,i}\P_{\bf U}(\tau_{0,l}\ge j)|\log\bar\alpha_{j-1}|+\sup_{l=1,\ldots,i}\P_{\bf U}(\tau_{0,l}\ge k+1)|\log\bar\alpha_{k}|
\]
It follows that $-\sum_{i=1}^k(A_{i,k}-A_{i,k+1})(A_{i,k}+A_{i,k+1})$ equals
\[
2\sum_{i=1}^k\sup_{l=1,\ldots,i}\P_{\bf U}(\tau_{0,l}\ge k+1)|\log\bar\alpha_{k}|\sum_{j=i}^k\sup_{l=1,\ldots,i}\P_{\bf U}(\tau_{0,l}\ge j)|\log\bar\alpha_{j-1}|
\]
\[
+\sum_{i=1}^k\left(\sup_{l=1,\ldots,i}\P_{\bf U}(\tau_{0,l}\ge k+1)|\log\bar\alpha_{k}|\right)^2.
\]
The second term vanishes since
\[
\sum_{i=1}^k\left(\sup_{l=1,\ldots,i}\P_{\bf U}(\tau_{0,l}\ge k+1)|\log\bar\alpha_{k}|\right)^2\le k(\log\bar\alpha_{k})^2\rightarrow0
\]
by our choice of $\bar\alpha_k,k\ge1$. Notice that the first term is bounded above by
\begin{align*}
&2|\log\bar\alpha_{k}|\sum_{i=1}^k\sum_{j=i}^k\sup_{l=1,\ldots,k}\P_{\bf U}(\tau_{0,l}\ge j)|\log\bar\alpha_{j-1}|\\
=&2|\log\bar\alpha_{k}|\sum_{j=1}^kj\sup_{l=1,\ldots,k}\P_{\bf U}(\tau_{0,l}\ge j)|\log\bar\alpha_{j-1}|.
\end{align*}
So, by our choice of $\bar\alpha_k,k\ge1$, a sufficient condition for this term to vanish is that
\begin{equation}\label{eq:general}
\sum_{j=1}^k\sup_{l=1,\ldots,k}\P_{\bf U}(\tau_{0,l}\ge j)=o(k).
\end{equation}
But observe that for any $\ell\in\{0,\ldots,k\}$,
\[
\tau_{0,\ell}:=\inf\{l\ge1:\tilde\xi^{(0)}_l=\tilde\xi^{(\ell)}_l=1\}\le \inf\{l\ge 1: \tilde\xi^{(i)}_l=1\,,i=0,\ldots, k \}=:{{\theta_k}}.
\]
Here, the formal definition of ${\theta_k}$ involves the coupling of $k+1$ renewal processes (or equivalently, of their backward recurrence time chains) instead of only two as we introduced above. However, the extension of this coupling to any finite (or even infinite) number of renewal processes is straightforward.  Thus, 
$$\sup_{\ell=1,\ldots,k}\P_{\bf U}(\tau_{0,\ell}\ge j)\leq \P_{\bf U}({\theta_k}\ge j).$$

 It only remains to show that $\sum_{j=1}^k\P_{\bf U}({\theta_k}\ge j)=o(k)$ under condition \eqref{eq:ok}.
 We recall that $q^\star_i:=\sup_{n\le i}q_n$, thus, by construction,
\[
\{{\theta_k}\ge j\}\subset\cap_{i=1}^{j-1}\{U_i\le q^\star_{k+i-1}\},
\]
since if $U_i\ge q^\star_{k+i-1}$ for some $i=1,\ldots, j-1$, then $U_i\ge q_{\xi^{(n)}_{i-1}}$ for $n=0, \ldots,k$, and therefore  $\xi^{(n)}_{i}=1$ for $n=0,\ldots,k$ (all the renewal processes merge at time $i<j$). Thus
\[
\sup_{\ell=1,\ldots,k}\P_{\bf U}(\tau_{0,\ell}\ge j)\leq \P_{\bf U}({\theta_k}\ge j)\le \P_{\bf U}\left(\cap_{i=1}^{j-1}\{U_i\le q^\star_{k+i-1}\}\right)=\prod_{i=k}^{k+j-2}q_i^\star
\]
concluding the proof of the proposition.
}

\end{proof}

\section{Discussion, interpretations and related models}\label{sec:discussion}
We begin this section with a discussion concerning Condition \eqref{eq:ok}. Then, we will explain how to get explicit bounds for the probability of coverage. Finally, we relate the model to  the firework process, a rumor process on $\N$. 
\subsection{About condition (\ref{eq:ok})} 

Condition \eqref{eq:general} that we used in the proof is strictly weaker than condition \eqref{eq:ok} as will show the two first examples below. We then give the calculations for the polynomial example, and conclude with an example that we could not handle with our technics.

\subsubsection{Doeblin for large $i$'s}

The condition \eqref{eq:general} that we obtained in our proofs is strictly weaker than condition \eqref{eq:ok}. 
Indeed, if $q_i=1$ for some $i$, then condition \eqref{eq:ok} cannot be satisfied, while condition \eqref{eq:general} may still be satisfied. Suppose there exists $i^\star\ge0$ such that $\epsilon\le q_k\le 1-\epsilon$ for any $k\ge i^\star$.  It can be checked that  $\alpha:=\inf_{i\ge0} Q^{i^{\star}+1}(i,0)>0$ (where $Q$ is the transition matrix of the backward recurrence time chain as before). This means that we can couple the $k+1$ Markov chain in such a way that each $i^\star+1$ steps, every coupled Markov chains have a positive probability, at least $\alpha$, to meet (and therefore merge) at state $0$. Thus
\[
\left(\sum_{j=1}^{k}\P_{\bf U}({\theta_k}\ge j)\right)^2\le \left(\sum_{j\ge1}\P_{\bf U}({\theta_k}\ge j)\right)^2\le \left(\sum_{k\ge1}(i^\star+1)(1-\alpha)^{k-1}\right)^2<\infty.
\]

\subsubsection{What happens when $q_i=0$ for some $i$?}
Something that we have not considered so far is the somewhat pathological case in which $q_i=0$ for some $i$. Obviously, if $q_0=0$ then ${\boldsymbol\xi}$ is a sequence of $1$'s. But if we suppose that $\bar i\ge1$ is the smallest integer such that $q_{\bar{i}}=0$, then ${\boldsymbol \xi}$ is a $\bar i+1$-steps Markov chain, with the property that the distance separating two consecutive $1$'s is at most $\bar i$. In particular, the backward recurrence time chains are  recurrent and aperiodic Markov chains  on a finite state space $\{0,\ldots,\bar i\}$, and as such, always satisfy \eqref{eq:general} with no need to assume \eqref{eq:ok}. Thus, Proposition \ref{coro:2} always holds in this case. 

\subsubsection{The polynomial case} 

As promised, we show here that  the polynomial case $q_i=1-\frac{1}{i^\beta}$, $i\ge2$ satisfies condition \eqref{eq:ok}. The process is positive recurrent iff 
$0<\beta<1$. By monotonicity, we have to show
\[
\sum_{j=1}^{k}\prod_{i=k}^{k+j-2}q_i=o(k).
\]
We have 
\begin{align*}
\prod_{i=k}^{k+j-2}\left(1-\frac{1}{i^\beta}\right)\leq \exp\left(-\sum_{i=k}^{k+j-2}\frac{1}{i^\beta}\right)\le\exp\left(-\int_{i=k}^{k+j-1}\frac{1}{x^\beta}dx\right).
\end{align*}
Thus, there exists  positive constant $c$ (depending on $\beta$) such that, for any sufficiently large $k$ we have
\begin{align*}
\sum_{j=1}^{k}\prod_{i=k}^{k+j-2}\left(1-\frac{1}{i^\beta}\right)\le e^{ck^{1-\beta}}\sum_{j=k+1}^{2k-1}e^{-cj^{1-\beta}}\le e^{ck^{1-\beta}}\int_{k}^{\infty}e^{-cx^{1-\beta}}dx.
\end{align*}
After a change of variables in order to reveal the upper incomplete Gamma function  $\Gamma(s,x):=\int_{x}^{\infty}t^{s-1}e^{-t}dt$, we get
\[
\int_{k}^{\infty}e^{-cx^{1-\beta}}dx= C.\Gamma\left(\frac{1}{1-\beta},{ck^{1-\beta}}\right)
\]
for some positive constant $C$  (depending on $\beta$).
Using now the well-known fact that $\Gamma(s,x)x^{1-s}e^{x}\rightarrow1$ as $x\rightarrow\infty$ when $s>0$, we have
\[
\sum_{j=1}^{k}\prod_{i=k}^{k+j-2}\left(1-\frac{1}{i^\beta}\right)ck^{-\beta}\le e^{ck^{1-\beta}}c^{-\beta/(1-\beta)}k^{-\beta}C\Gamma\left(\frac{1}{1-\beta},{ck^{1-\beta}}\right)\rightarrow c^{-\beta/(1-\beta)}.
\]
In other words, condition \eqref{eq:ok} holds for any $\beta\in(0,1)$.
%
%
%
 
\subsubsection{An example that does not fit our conditions}

We conclude with an example that does not fit our conditions. Consider the case $q_i=\left(\frac{i+1}{i+2}\right)^\beta$, $i\ge1$ with $\beta>1$. Using   monotonicity,
\[
\sum_{j=1}^{k}\prod_{i=k}^{k+j-2}q_i^\star=\sum_{j=1}^{k}\prod_{i=k}^{k+j-2}q_i=\sum_{j=1}^{k}\left(\frac{k+1}{k+j}\right)^{\beta}=(k+1)^\beta\sum_{j=k+1}^{2k}\frac{1}{j^\beta}
\]
which is of order $k$, and therefore not $o(k)$.  Using condition \eqref{eq:general} would give exactly the same calculations. 
 This case is therefore not covered by our results.

\subsection{Explicit  bounds for the probability of coverage}

The upper bound obtained in the above proofs together with Lemma \ref{lem:concentration1} yield
\[
\left(1+\sum_{k\ge1}\prod_{i=0}^{k-1}\alpha_i^{\P(\xi_{i+1}=1)}e^{C_k}\right)^{-1}\le \P(\mathcal{A})\leq \left(1+\sum_{k\ge1}\prod_{i=1}^{k-1} \alpha_i^{\P(\xi_{i+1}=1)}\right)^{-1}
\]
where $C_k$ is defined by \eqref{eq:ok}.
So explicit bounds rely on the ability to obtain explicit bounds on $C_k$ and $\P(\xi_{i}=1)$. These are easy to obtain in the Markov case (the case in which $q_l=q_1$ for all $l\ge1$) for instance, where it is well-known that (denoting $\epsilon:=\min\{q_0,q_1,1-q_0,1-q_1\}$)
\[C_k\leq \sum_{j=1}^{k}(1-\epsilon)^{j-1}\]
and
\[\P(\xi_i=1)=\frac{1}{1-q_1+q_0}\left(q_0+(1-q_1)(q_1-q_0)^{i+1}\right).\]
In the general non-Markovian case, several bounds can be found in the literature for the renewal probability $\P(\xi_i=1)$, but they are generally asymptotic and we do not pursue here these considerations. 

Let us only mention that in the case where the $q_i$'s are monotonically increasing (see two examples in the preceding subsection), the measure of the renewal process has the ``monotonicity'' property (see e.g. \citeasnoun{georgii/haggstrom/maes/2001}) allowing to use the FKG inequality  $\E\prod_{i=0}^{k-1}\alpha_i^{\xi_{i+1}}\ge \prod_{i=0}^{k-1}\E\alpha_i^{\xi_{i+1}}$. Thus

\[
\P(\mathcal{A})\leq\left(1+\sum_{n\ge1}\prod_{i=0}^{n-1}[1-\P(\xi_{i+1}=1)(1-\alpha_i)]\right)^{-1}.
\]
This upper bound is tight because we know that the i.i.d. case satisfies the monotonicity condition and by Corollary \ref{coro:i.i.d} it reaches this upper bound.

\subsection{Extended Reverse Firework processes }\label{sec:FP} When ${\boldsymbol \xi}$ is identically $1$ or i.i.d., the  model described in Section \ref{sec:model} has been studied under the name Firework Process (FP) by
\citeasnoun{junior/machado/zuluaga/2011} and
\citeasnoun{gallo/garcia/junior/rodriguez/2014} among others. {These
works also studied another rumor process called Reverse
Firework Process (RFP) which, as observed in \citeasnoun{gallo/garcia/junior/rodriguez/2014} (Section 5, Item (1)), bears a strong relationship with the FP. 

Our coverage model is an extension of the FP in which we use a  renewal sequence ${\boldsymbol \xi}$. The corresponding extension of the RFP can be described as follows. We start with an undelayed renewal process $\boldsymbol{\zeta}$ and define the random set of integers $\mathcal{S}:=\{i\ge0:\zeta_i=1\}$. Associate to each site $i\ge0$ an independent copy $L_i$ of some $\N$-valued random variable $L$. At step 0,  there is some information at the origin, let $B_0:=\{0\}$ and at step $n\ge1$, define recursively the set of \emph{newly} informed individuals 
\[B_n:=\{i\in\mathcal{\loca}:\,\,\{i-L_i,\ldots,i-1\}\cap B_{n-1}\neq\emptyset\}\setminus \cup_{j=0}^{n-1}B_{j}.\]
Pictorially, an individual is informed if its radius, sent towards its left, covers an informed individual. {Thus $i\in\cup_jB_j$ if and only if ``$\zeta_i=1$ and  there exists $l<i$ such that $Z_l=1$ and $L_i\ge i-l$''. Recalling \eqref{eq:conta}, we observe that for any $n\ge1$, $Z_i,i=0,\ldots,n$  has the same distribution as the dual $Y^{(n)}_i,i=0,\ldots,n$ when $Y^{(n)}_0=1$, that is, when the later is undelayed. }

The moral of what we said above is that the reverse firework process has the same distribution as the dual of the firework process. The following result is therefore a direct consequence of Proposition \ref{theo:rfp}.


\begin{coro}\label{transience}
The renewal reverse firework process survives if, and only if, $\prod_{k\ge0}\alpha_k=0$.
\end{coro}

The same result was obtained in \citeasnoun{junior/machado/zuluaga/2011} (case where
$\xi_k=1$ for all $k\ge1$) and in
\citeasnoun{gallo/garcia/junior/rodriguez/2014} (case where
${\boldsymbol \xi}$ is an i.i.d. sequence). 
Theorem 3 in \citeasnoun{gallo/garcia/junior/rodriguez/2014} gives
much more information concerning the reverse firework process
(central limit theorem for instance). {We will not discuss these results here, but mention that since these properties follow
directly from the fact that the FP is a renewal process, they are all
inherited by the reverse firework process since, by Proposition \ref{theo:rfp},
the RFP is also a renewal process. }

%

\subsection{Queueing systems}
The renewal process ${\bf \xi}$ which dictates the origin of each intervals can be seen as an arrival process of individuals in a queueing system with infinitely many servers, and the length of the intervals associated to each mark of the renewal process can be seen as the service time of the corresponding individual. {Therefore, we can view our process as  a discrete time GI/G/$\infty$ queue with infinite expected service time, a regime that is rarely studied  by the queueing theory community. } \\

\noindent{\bf Acknowledgments} {Sandro Gallo thanks FAPESP 2015/09094-3 and research fellowship CNPq (312315/2015-5). Nancy Garcia thanks FAPESP 2014/26419-0 and research fellowship CNPq 302598/2014-6. This article was produced as part of the activities of FAPESP  Research, Innovation and Dissemination Center for Neuromathematics (grant \#2013/07699-0 , S\~ao Paulo Research Foundation) and Edital Universal CNPq (480108/2012-9 and 462064/2014-0).}

\bibliography{sandrobibli}

\appendix
\section{About renewal sequences }\label{sec:renewal}

Here we list some  basic and well-known facts concerning renewal sequences that we use in our proofs. Let ${\boldsymbol\eta} = \{\eta_i, i \in \N \}$ be a renewal sequence with inter-arrival time $T$.

If it starts with $\eta_0=1$, it is called undelayed.  It is well-known that this sequence is positive recurrent if, and only if, $\E T<\infty$, and, as far as $T$ has a proper distribution, the Renewal Theorem states that $\lim_n \P(\eta_n=1)=1/\E T$ (use $1/\infty=0$). 

We can also suppose that ${\boldsymbol\eta}$ is delayed, that is, there is a $\N$-valued random variable $T_0$ which gives the time elapsed until the first occurrence of a $1$. That is, for $i\ge0$, $\eta_i=1$ if there exists $n\ge1$ such that $\sum_{k=0}^nT_k=i$, and we recover an undelayed sequence by putting $T_0\equiv 0$. As long as $T_0$ is a proper random variable, the Renewal Theorem also holds, and the limit is the same, no matter the delay, that is $\lim_n \P(\eta_n=1)=1/\E T$. 

 {There are two other important observations concerning renewal sequences, delayed or not, that we  use in the present paper. 
\begin{enumerate}[(i)]
\item  For $k\in\{1,\ldots,n\}$, ${\boldsymbol \eta}$ has conditional probabilities 
\begin{equation}\label{eq:altern_def}
\mathbb{P}(\eta_n=1|\eta_0^{n-1}=a_0^{n-1})=\P(T= k|T\ge k)
\end{equation}
for any string $a_0^{n-1}$  of symbols of $\{0,1\}$  such that $a_{n-k}=1$ and $a_{n-k+i}=0$ for $i=1,\ldots,k-1$. In other words, the conditional probability of time $n$ given the whole past  depends only on the distance since the last occurrence of a $1$ in the realisation, that is, 
\[
\E({\bf1}\{\eta_n=1\}|\eta_0^{n-1})=\E({\bf1}\{\eta_n=1\}|\ell_1(\eta_0^{n-1})).
\]
where $\ell_1(\eta_0^{i}):=\inf\{j\ge0: \,\eta_{i-j}=1\}$. This is  a property that characterises $1$ as a renewal event for the sequence. Naturally,  the sequence ${\boldsymbol \eta}$ can be constructed by applying recursively the conditional probabilities \eqref{eq:altern_def}, beginning with $\eta_0=1$ if it is undelayed, and from $T_0$ if it is delayed. So, together with the distribution of $T_0$, these conditional probabilities define the process univocally. 

\item  ${\boldsymbol \eta}$ has the following reversibility property: 
for any  $n\ge1$, the law of the sequence $\zeta^{(n)}_i:=\eta_{n-i},i=0,\ldots,n$ is that of a renewal process with the same inter-arrival distribution. Moreover, if ${\boldsymbol \eta}$ is undelayed, on the event that $\eta_n=1$, the sequence $\zeta^{(n)}_i,i=0,\ldots,n$ has the same distribution as $\eta_i,i=0,\ldots,n$. These property are easy to obtain when we observe that these sequences are constructed \emph{via} a concatenation of blocks of the form $0\ldots01$, of independent size equally distributed with $T$. 
\end{enumerate}}

\section{Proof of Lemma \ref{lem:concentration1}} \label{sec:concentration}

The proof will be done in three parts. First we explain the martingale difference method used to obtain concentration inequalities. At some point, we will need to estimate the probability of discrepancy between coupled delayed renewal processes, and thus, the second part of the proof is the explicit construction of a coupling of these processes. The proof of the lemma is concluded in a third step. 

\begin{notalert}
For notational simplicity, we will translate the indexes of one unit along the present section, to study $\E\prod_{i=1}^{k}\alpha_i^{\xi_i}$  instead of $\E\prod_{i=0}^{k-1}\alpha_i^{\xi_{i+1}}$.
\end{notalert}

\paragraph{The method of martingale difference.}\cite[Section 4]{mcdiarmid/1989} Let  $g(\xi_1,\ldots,\xi_{k}):=\sum_{i=1}^{k} \xi_i\log\alpha_i$. We will upper-bound 
$\E(e^{g-\E g})$. Let, for $i=1,\ldots,k$
\[\Delta_i:=\E(g|\mathcal{F}_{1}^{i})-\E(g|\mathcal{F}_{1}^{i-1}),
\]
where, for $i\ge1$, $\mathcal{F}_1^i$ is the $\sigma$-algebra generated by $\xi_j,j=1,\ldots,i$ and $\mathcal{F}_1^0$ is the trivial $\sigma$-algebra. These quantities sum telescopically in $i$ to $\sum_{i=1}^{k}\Delta_i(\xi_{1}^{i})=g(\xi_{1}^{k})-\E g(\xi_1^k)$.
Since $\mathcal{F}_1^{i-1}\subset\mathcal{F}_1^{i}$, we have $\E (\Delta_i|\mathcal{F}_{1}^{i-1})=
0$ which means that  $\Delta_i$, $i\ge1$ forms a martingale difference sequence. If  there exists, for any $i\ge1$, a finite real number $d_i$ such that $|\Delta_i|\leq d_i$ a.s., we can use the  Azuma-Hoeffding inequality \cite[proof of Lemma 4.1]{azuma/1967} which states
\begin{equation}\label{eq:hoef}
\E e^{g-\E g}\leq e^{\frac{1}{8}\sum_{i=1}^kd_i^2}.
\end{equation}
To upper bound $\sum_id_i^2$, let us compute, for $i=1,\ldots,k-1$
\begin{align*}
\Delta_i(\xi_{1}^{i})=&\sum_{u_{i+1}^{k}\in\{0,1\}^{k-i}}\P(u_{i+1}^{k}|\xi_{1}^{i})g(\xi_1^iu_{i+1}^{k})-\sum_{u_{i}^{k}\in\{0,1\}^{k-i+1}}\P(u_{i}^{k}|\xi_{1}^{i-1})g(\xi_1^{i-1}u_{i}^{k})\\
\leq&|\sum_{u_{i+1}^{k}}\P(u_{i+1}^{k}|\xi_{1}^{i-1}1)g(\xi_1^{i-1}1u_{i+1}^{k})-\sum_{u_{i+1}^{k}}\P(u_{i+1}^{k}|\xi_{1}^{i-1}0)g(\xi_1^{i-1}0u_{i+1}^{k})|
\end{align*}
where we used the convention that $\xi_1^0=\emptyset$ and the notation of concatenation between strings $a_i^jb_k^l=(a_i,\ldots,a_j,b_k,\ldots,b_l)$. A similar computation yields $\Delta_k(\xi_{1}^{k})\leq\left|\log\alpha_k\right|$.

Therefore,
$$\Delta_i(\xi_{1}^{i})\leq \sum_{j=0}^{k-i}D_{i,i+j}(\xi_{1}^{i})|\log\alpha_{i+j}|=(D(\xi_1^k)L)_i\,,\,\,i=1,\ldots,k$$
where $D(\xi_{1}^{n})$ is the upper triangular $k\times k$ matrix defined by
\begin{align*}
D_{i,i+j}(\xi_{1}^{i})&:=\sum_{u_{i+1}^{k},v_{i+1}^{k}\in\{0,1\}^{k-i}}\mathcal{Q}(u_{i+1}^{k},v_{i+1}^{k}|1\xi_{1}^{i-1}1,1\xi_{1}^{i-1}0){\bf1}\{u_{i+j}\neq v_{i+j}\}\\&=\mathcal{Q}(u_{i+j}\neq v_{i+j}|1\xi_{1}^{i-1}1,1\xi_{1}^{i-1}0), \quad \mbox{for  } j=0,\ldots,k-i,
\end{align*}
 $L$ is the $k\times1$ matrix (column  vector) with entries $L_{i,1}=|\log\alpha_i|,\,i=1,\ldots,k$ and finally $\mathcal{Q}((\cdot,\cdot)|1\xi_{1}^{i-1}1,1\xi_{1}^{i-1}0)$ denotes the law of a coupling between two discrete renewal sequence having the same inter-arrival distribution and starting from the different configurations $1\xi_{1}^{i-1}1$ and $1\xi_{1}^{i-1}0$. 

\paragraph{Coupling and conclusion of the proof of the Lemma.}
 Let $\ell(1\xi_{1}^{i})$ denote the smaller integer $k$ such that $\xi^i_{i-k+1}=0^k$, where $0^k=(0,\ldots,0)$ denotes the string of $k$ consecutive $0$'s. Observe that $\ell(1\xi_{1}^{i})\leq \ell(10^i)=i$ (this is one of the main differences with the case of \citeasnoun{chazottes/collet/kulske/regig/2007}). Then, for $i\ge1$, $D_{i,i+j}(\xi_{1}^{i})$ is equal to the probability that two coupled renewal processes, one undelayed (the one starting with $1\xi_{1}^{i-1}1$), and the other with delay $\ell(1\xi_{1}^{i-1}0)\ge1$ (the one starting with $1\xi_{1}^{i-1}0$), disagree at time $j$. 
Recall the coupling that we define before the statement of the lemma. 
%
Using this coupling we have the upper bound,
\[
D_{i,i+j}(\xi_{1}^{i})=\P_{\bf U}\left(\tilde\xi^{(0)}_j\neq\tilde\xi^{(\ell(\xi_{1}^{i}))}_j\right)\leq \P_{\bf U}(\tau_{0,\ell(\xi_{1}^{i})}\ge j) \le \sup_{\ell=1,\ldots,i}\P_{\bf U}(\tau_{0,\ell}\ge j)
\]
where the last inequality follows  taking the supremum over all the possible values of $\ell(\xi_{1}^{i})$ for any $i=1,\ldots,k$.
{In view of \eqref{eq:hoef}, we now take
\[
d_i:=( DL)_i=\sum_{j=i}^k\sup_{\ell=1,\ldots,i}\P_{\bf U}(\tau_{0,\ell}\ge j)|\log\alpha_j|
\]
and thus, 
\[
\sum_{i=1}^kd_i^2\leq \sum_{i=1}^k\left(\sum_{j=i}^k\sup_{\ell=1,\ldots,i}\P_{\bf U}(\tau_{0,\ell}\ge j)|\log\alpha_j|\right)^2.
\]
Recalling that we translated the indexes in the beginning of the proof, what we proved, from \eqref{eq:hoef}, is that
\[
\E\left(e^{\sum_{i=0}^{k-1} \xi_i\log\alpha_i-\E\sum_{i=0}^{k-1} \xi_i\log\alpha_i}\right)\leq e^{\frac{1}{8}\sum_{i=1}^k\left(\sum_{j=i}^k\sup_{\ell=1,\ldots,i}\P_{\bf U}(\tau_{0,\ell}\ge j)|\log\alpha_{j-1}|\right)^2}
\]
and therefore
\begin{equation}\label{eq:mais_geral}
\E\prod_{i=0}^{k-1}\alpha_i^{\xi_i}\leq \prod_{i=0}^{k-1}\alpha_i^{\P(\xi_i=1)}e^{\frac{1}{8}\sum_{i=1}^k\left(\sum_{j=i}^k\sup_{\ell=1,\ldots,i}\P_{\bf U}(\tau_{0,\ell}\ge j)|\log\alpha_{j-1}|\right)^2}.
\end{equation}
This concludes the proof of the lemma. }

\end{document}